\documentclass[12pt,reqno]{amsart}
\usepackage{url}
\usepackage{graphics}
\usepackage{color}

\newtheorem{Theorem}{Theorem}[section]
\newtheorem{Lemma}[Theorem]{Lemma}
\newtheorem{Corollary}[Theorem]{Corollary}
\newtheorem{Proposition}[Theorem]{Proposition}
\newtheorem{Remark}[Theorem]{Remark}

\newtheorem{Example}[Theorem]{Example}

\newtheorem{Question}[Theorem]{Question}

\def\qed{\ifhmode\textqed\fi
	\ifmmode\ifinner\hfill\quad\qedsymbol\else\dispqed\fi\fi}
\def\textqed{\unskip\nobreak\penalty50
	\hskip2em\hbox{}\nobreak\hfill\qedsymbol
	\parfillskip=0pt \finalhyphendemerits=0}
\def\dispqed{\rlap{\qquad\qedsymbol}}

\def\p{\mathfrak{p}}

\def\m{\mathfrak{m}}

\def\v{\textup{v}}
\def\ZZ{\mathbb{Z}}
\def\soc{\textup{soc}}

\def\FF{\mathbb{F}}

\def\HS{\textup{HS}}
\def\set{\textup{set}}
\def\pd{\textup{proj\,dim}}
\def\depth{\textup{depth}}
\def\reg{\textup{reg}}
\def\n{\mathfrak{n}}
\def\lcm{\textup{lcm}}
\def\supp{\textup{supp}}

\def\Tor{\textup{Tor}}
\def\Ass{\textup{Ass}}
\def\S{\mathcal{S}}

\begin{document}
	
	\title{The homological shift algebra\\ of a monomial ideal}	
	\author{Antonino Ficarra, Ayesha Asloob Qureshi}
	
	\dedicatory{Dedicated with deep gratitude to the memory of Professor J\"urgen Herzog,\\ inspiring mathematician and master of monomials}
	
	\address{Antonino Ficarra, Departamento de Matem\'{a}tica, Escola de Ci\^{e}ncias e Tecnologia, Centro de Investiga\c{c}\~{a}o, Matem\'{a}tica e Aplica\c{c}\~{o}es, Instituto de Investiga\c{c}\~{a}o e Forma\c{c}\~{a}o Avan\c{c}ada, Universidade de \'{E}vora, Rua Rom\~{a}o Ramalho, 59, P--7000--671 \'{E}vora, Portugal}
	\email{antonino.ficarra@uevora.pt}\email{antficarra@unime.it}
	
	\address{Ayesha Asloob Qureshi, Sabanci University, Faculty of Engineering and Natural Sciences, Orta Mahalle, Tuzla 34956, Istanbul,
		Turkey}
	\email{ayesha.asloob@sabanciuniv.edu}\email{aqureshi@sabanciuniv.edu}
	
	\thanks{
	}
	
	
	
	\begin{abstract}
		Let $S=K[x_1,\dots,x_n]$ be the polynomial ring over a field $K$, and let $I\subset S$ be a monomial ideal. In this paper, we introduce the $i$th \textit{homological shift algebras} $\text{HS}_i(\mathcal{R}(I))=\bigoplus_{k\ge1}\text{HS}_i(I^k)$ of $I$. If $I$ has linear powers, these algebras have the structure of a finitely generated bigraded module over the Rees algebra $\mathcal{R}(I)$ of $I$. Hence, many invariants of $\text{HS}_i(I^k)$, such as depth, associated primes, regularity, and the $\v$-number, exhibit well behaved asymptotic behavior. We determine several families of monomial ideals $I$ for which $\text{HS}_i(I^k)$ has linear resolution for all $k\gg0$. Finally, we show that $\text{HS}_i(I^k)$ is Golod for all monomial ideals $I\subset S$ with linear powers and all $k\gg0$.
	\end{abstract}
	
	\maketitle
	\vspace*{-1cm}
	\section*{Introduction}
	
	Let $S=K[x_1,\dots,x_n]$ be the polynomial ring over a field $K$, and let $I\subset S$ be a monomial ideal. The concept of the \textit{homological shift ideal} was introduced in 2020 by Herzog \textit{et al.} \cite{HMRZ021a}, although it first appeared implicitly in 2005 in the book by Miller and Sturmfels \cite[Theorem 2.18]{MS}. In that theorem, using the language that we will explain shortly after, the authors proved that the first homological shift ideal of an equigenerated Borel ideal has a linear resolution.
	
	Given ${\bf a}=(a_1,\dots,a_n)\in\ZZ_{\ge0}^n$, we set ${\bf x^a}=x_1^{a_1}\cdots x_n^{a_n}$. The $i$th \textit{homological shift ideal} $\HS_i(I)$ of $I$ is defined as the monomial ideal generated by the monomials ${\bf x^a}$ such that ${\bf a}$ is an $i$th \textit{multigraded shift} of $I$. That is,
	$$
	\HS_i(I)\ =\ ({\bf x^a}\ :\ \beta_{i,{\bf a}}(I)\ne0),
	$$
	where $\beta_{i,{\bf a}}(I)=\dim_K\Tor^S_i(K,I)_{\bf a}$. It is clear that $\HS_0(I)=I$ and $\HS_i(I)=0$ if $i\notin\{0,\dots,\pd(I)\}$. See also \cite{BQ23,BJT019,FPack1,F-SCDP,FQ,HMRZ021b,LW,TBR24}.
	
	The ideals $\HS_i(I)$ encode some of the information provided by the multigraded syzygies of $I$. It should be noted, however, that not all the multigraded shifts of $I$ correspond to minimal generators of $I$. This fact does not hold if $I$ has a linear resolution. This observation suggests that the algebraic properties of $\HS_i(I)$ are closer to those enjoyed by $I$ when $I$ has a linear resolution. We say that $I$ has \textit{homological linear resolution} if $\HS_i(I)$ has linear resolution for all $i$.
	
	In this vein, it was conjectured in 2012 by Bandari, Bayati and Herzog that the homological shift ideals of polymatroidal ideals are polymatroidal. Since polymatroidal ideals have a linear resolution, these ideals would constitute a family of monomial ideals with homological linear resolution. This conjecture is still open, however, it is supported by several partial results \cite{Ban24,Bay019,Bay2023,F2,FH2023,HMRZ021a}.
	
	In \cite{CF1,CF2}, Crupi and the first author conjectured that the homological shift ideals of each power of the cover ideal of a Cohen-Macaulay very well-covered graph have a linear resolution, and they proved this in some special cases.
	
	For a homogeneous ideal $I\subset S$, let $\alpha(I)$ denote the \textit{initial degree} of $I$, that is, the smallest degree of a generator of $I$. It is known that the powers of polymatroidal ideals are polymatroidal \cite[Theorem 12.6.3]{HH2011}. Thus, if the Bandari-Bayati-Herzog conjecture, or the Crupi-Ficarra conjecture is true, then for the monomial ideals $I\subset S$ considered in these conjectures, $\HS_i(I^k)$ would have linear resolution for all $i$ and $k\ge1$. In particular, the Castelnuovo-Mumford regularity $\reg\,\HS_i(I^k)$ would be a linear function of the form $\alpha(\HS_i(I^k))=\alpha(I)k+i$, for all $k\gg0$.
	
	Thus, we were led to ask whether $\reg\,\HS_i(I^k)$ becomes a linear function for $k\gg0$ for a monomial ideal $I\subset S$ with \textit{linear powers}. Recall that a monomial ideal $I\subset S$ has \textit{linear powers} if $I^k$ has a linear resolution for all $k\ge1$.  In this paper, we show that this is indeed the case. This phenomenon arises due to the fact that the $K$-algebra
	$$
	\HS_i(\mathcal{R}(I))\ =\ \bigoplus_{k\ge1}\HS_i(I^k)
	$$
	is a finitely generated bigraded module over the Rees algebra $\mathcal{R}(I)$ of $I$, provided that $I$ has linear powers. We call $\HS_i(\mathcal{R}(I))$ the $i$th \textit{homological shift algebra} of $I$. In this paper, we deeply study the homological shift algebras of a monomial ideal.
	
	In Section \ref{sec1}, we introduce the $i$th homological shift algebra of a monomial ideal $I$ and investigate when this algebra has the structure of a $\mathcal{R}(I)$-module. For this aim, one must require that $I\cdot\HS_i(I^k)\subseteq\HS_i(I^{k+1})$ for all $k\ge1$.  We find examples of monomial ideals such that $I\cdot\HS_i(I^k)\not\subseteq\HS_i(I^{k+1})$ for all $k\ge1$. On the other hand, we show in Theorem \ref{Thm:HS(R(I))} that the opposite inclusion $\HS_i(I^{k+1})\subseteq I\cdot\HS_i(I^k)$ holds for all $k\gg0$. Under the assumption that $I$ has linear powers, we prove in Theorem \ref{Thm:HS-lp} that $\HS_i(\mathcal{R}(I))$ it is a finitely generated bigraded $\mathcal{R}(I)$-module. In Proposition \ref{Prop:HS-n-1}, we determine some conditions on $I$ that ensure that $\HS_{n-1}(I^{k+1})=I\cdot\HS_{n-1}(I^k)$ for all $k\gg0$. We apply this result to edge ideals $I(G)$ of graphs $G$ whose induced odd cycles satisfy certain distance conditions. Hence, if $G$ is bipartite or unicyclic, then $\HS_{n-1}(I(G)^{k+1})=I(G)\cdot\HS_{n-1}(I(G)^k)$ for all $k\gg0$.
	
	Next, in Section \ref{sec2}, we address the main question that motivated this paper in the first place. Namely, we prove in Theorem \ref{Thm:AsymHS} that $\reg\,\HS_i(I^k)$ is a linear function for all $k\gg0$, provided that $I$ has linear powers. Under the same assumption on $I$, we also show that the depth and the associated primes of $\HS_i(I^k)$ stabilize and that the $\v$-number of $\HS_i(I^k)$ is an eventually linear function for all $k\gg0$. The $\v$-number of a homogeneous ideal $I$ is a new invariant introduced firstly in \cite{CSTVV20} which measures the homogeneous complexity of the primary decomposition of $I$. See also \cite{Conca23,F2023,FS2,FSPack,FSPackA,Ghosh24} and the references therein. It is proved in \cite[Theorem 4.1]{FS2} that $\lim_{k\rightarrow\infty}\v(I^k)/k=\alpha(I)$ for any graded ideal. Thus one could expect that $\lim_{k\rightarrow\infty}\v(\HS_i(I^k))/k=\alpha(I)$ if $\HS_i(I^k)\ne0$ for all $k\gg0$. This equation holds for $i=0$ or when $I$ is equigenerated by \cite[Theorem 1.9(2)]{Ghosh24}. To demonstrate our theory, we compute the asymptotic invariants of the homological shift algebras of some distinguished families of monomial ideals.
	
	In Section \ref{sec3}, we turn to the monomial ideals whose all powers have linear quotients. A powerful technique to show that an equigenerated monomial ideal $I\subset S$ has linear resolution is to show that $I$ has linear quotients \cite[Theorem 8.2.15]{HH2011}. Herzog and the first author proved that if $I$ is an equigenerated monomial ideal with linear quotients, then $\HS_1(I)$ has linear quotients, and so a linear resolution. It is an open question whether this result remains true if we only assume that $I$ has linear resolution. Supporting this expectation, we prove in Theorem \ref{Thm:HS-1} that $\HS_1(I)$ has linear relations if $I$ has linear relations. This result cannot be extended to the higher homological shift ideals. Indeed, in Example \ref{Ex:I=B(125,333)} we present an equigenerated strongly stable ideal $I=B(x_1x_2x_5,x_3^3)$ such that $\HS_2(I^k)$ has no linear relations, and therefore no linear resolution and no linear quotients, for all $k\ge1$. This ideal has two Borel generators and is generated in degree three. This example also shows that Proposition \ref{Prop:HSiBorel}, as well as the result of Miller and Sturmfels \cite[Theorem 2.18]{MS}, cannot be extended to all strongly stable ideals. We say that $I$ has \textit{eventually homological linear powers} if $I$ has linear powers and $\HS_i(I^k)$ has linear resolution for all $k\gg0$. Not all monomial ideals with linear powers have eventually homological linear powers as Example \ref{Ex:I=B(125,333)} shows. Theorem \ref{Thm:familiesHLP} presents several classes of ideals with eventually homological linear powers.
	
	Finally, in Section \ref{sec4}, we consider the asymptotic Golodness of $\HS_i(I^k)$, when $I$ has linear powers. By an influential result of Herzog and Huneke \cite{HHun}, the powers $I^k$ of the graded ideal $I\subset S$ are Golod from $k\ge2$ on. In view of this result, it is natural to expect that if $I$ has linear powers, then $\HS_i(I^k)$ is a Golod ideal for all $k\gg0$. In Theorem \ref{Thm:HS-i-Golod}, we show that this is indeed the case. This result is a consequence of the more general Theorem \ref{Thm:M_k-Golod}, which in turn is a generalization of a result of Herzog, Welker and Yassemi \cite{HWY}.

	\section{The homological shift algebra}\label{sec1}
	
	Let $I\subset S=K[x_1,\dots,x_n]$ be a monomial ideal and let $\mathcal{G}(I)$ be its minimal monomial generating set. For systematic reasons, we set $I^0=S$. We want to investigate the asymptotic behavior of the homological invariants of the ideals $\HS_i(I^k)$ for a fixed $i$ and all $k\gg0$. For this aim, we consider the following natural $K$-algebra. The $i$th \textit{homological shift algebra} of $I$ is the $K$-algebra defined as
	$$
	\HS_i(\mathcal{R}(I))\ =\ \bigoplus_{k\ge1}\HS_i(I^k).
	$$
	
	Notice that $S\oplus\HS_0(\mathcal{R}(I))$ is simply the Rees algebra $\mathcal{R}(I)=\bigoplus_{k\ge0}I^k$ of $I$. Our goal is to investigate under which conditions the $K$-algebra $\HS_i(\mathcal{R}(I))$ has a natural structure of a bigraded $\mathcal{R}(I)$-module.
	
	We regard $\mathcal{R}(I)$ as a bigraded ring with $\mathcal{R}(I)_{(d,k)}=(I^k)_d$ as $(d,k)$th bigraded component, and we put $\HS_i(\mathcal{R}(I))_{(d,k)}=\HS_i(I^k)_d$.
	
	For a bigraded $R$-module $M=\bigoplus_{d,k\ge0}M_{(d,k)}$ over a graded ring $R=\bigoplus_{k\ge0}R_k$, we set $M_{(*,k)}=\bigoplus_{d\ge0}M_{(d,k)}$.
	
	Notice that $\HS_i(\mathcal{R}(I))$ is a $\mathcal{R}(I)$-module if and only if
	$$
	\mathcal{R}(I)_1\cdot\HS_i(\mathcal{R}(I))_k\ \subseteq\ \HS_i(\mathcal{R}(I))_{k+1}
	$$
	for all $k\ge1$. So, we should have $I\cdot\HS_i(I^k)\subseteq\HS_i(I^{k+1})$ for all $k\ge1$. Of course, this is not the case in general. In the next example, we present a monomial ideal $I$ such that $I\cdot\HS_i(I^{k})\not\subseteq\HS_i(I^{k+1})$ for all $k\ge1$ and some $i$.
	
	Let $\m=(x_1,\dots,x_n)$. Recall that the \textit{socle} of $I\subset S$ is the ideal
	$$
	\soc(I)\ =\ ({\bf x^a}\ :\ ((I:\m)/I)_{\bf a}\ne0).
	$$
	By \cite[Proposition 2.13]{HMRZ021a} we have $\HS_{n-1}(I)=x_1x_2\cdots x_n\cdot\soc(I)$. In particular, $\soc(I)$ is generated by all monomials $u\notin I$ such that $x_ju\in I$ for all $1\le j\le n$.
	\begin{Example}\label{Ex:HS-Not-Inclusion}
		Let $I=(x^2,y^2,xyz)\subset K[x,y,z]$. Then $I\cdot\HS_2(I^{k})\not\subseteq\HS_2(I^{k+1})$ for all $k\ge1$. Indeed, $u_k=x^3y^{2k+1}z^2\in (I\cdot\HS_2(I^k))\setminus\HS_2(I^{k+1})$ for all $k\ge1$.
	\end{Example}
	\begin{proof}
		Since $\HS_2(I^k)=xyz\cdot\soc(I^k)$, it follows that $u_k\in I\cdot\HS_2(I^k)$ if and only if
		$g_k=u_k/(xyz)=x^2y^{2k}z\in I\cdot\soc(I^k)$. Since $xyz\in I$, it is enough to show that $v_k=xy^{2k-1}\in\soc(I^k)$. The only generator of $I$ which divides $v_k$ is $y^2$. Since $\deg_y(v_k)=2k-1$, it follows that $v_k\notin I^k$. Now we show that $xv_k,yv_k,zv_k\in I^k$. Indeed,
		\begin{align*}
			xv_k\ &=\ x^2y^{2(k-1)}\cdot y\in I^k,\\
			yv_k\ &=\ y^{2k}\cdot x\in I^k,\\
			zv_k\ &=\ (xyz)y^{2(k-1)}\in I^k.
		\end{align*}
		Hence $v_k\in\soc(I^k)$ and so $u_k\in I\cdot\HS_2(I^k)$.\smallskip
		
		It remains to show that $u_k\notin\HS_2(I^{k+1})$. Since $\HS_2(I^{k+1})=xyz\cdot\soc(I^{k+1})$, we must show that $g_k=x^2y^{2k}z\notin\soc(I^{k+1})$. For this aim, it is enough to show that
		\begin{equation}\label{eq:soc(I^{k+1})}
			\soc(I^{k+1})\ =\ (x^{p}y^{q}\ :\ p+q=2(k+1),\ \textit{and $p$, $q$ are odd}).
		\end{equation}
		We have $I:(x)=(x,y^2,yz)$, $I:(y)=(x^2,y,xz)$ and $I:(z)=(x^2,y^2,xy)$. Hence $I:(z)$ is contained in $I:(x)$ and $I:(y)$. Recall that for a monomial ideal $J$ and a variable $x_i$, we have $J^k:(x_i)=J^{k-1}(J:(x_i))$ for all $k\ge1$. Thus $I^k:(z)$ is a subset of $I^k:(x)$ and $I^k:(y)$ for all $k\ge1$. Let $\m=(x,y,z)$. Then
		$$
		I^k:\m=(I^k:(x))\cap (I^k:(y))\cap (I^k:(z))=I^k:(z)=I^{k-1} (I:(z))=I^{k-1}(x,y)^2.
		$$
		Notice that $I^k:\m=I^{k-1}(x,y)^2$ contains all monomials of degree $2k$ in the variables $x,y$, and for any monomial $u\in(I^k:\m)$ with $z \in \supp(u)$, we have $\deg_x(u)+\deg_y(u)$ at least $2k$. Hence any monomial of $(I^k:\m)$ with $z \in \supp(u)$ is not a minimal generator. Since $\alpha(I^k:\m)=\alpha(I^{k-1}(x,y)^2)=2k$, we see that $I^k:\m=(x,y)^{2k}$ and
		$$
		(I^{k+1}:\m)/I^{k+1}=(x^{p}y^{q}+I^{k+1}\ :\ p+q=2(k+1),\ \textit{and p,q are odd}).
		$$
		This shows that equation (\ref{eq:soc(I^{k+1})}) indeed holds and so $g_k\notin\soc(I^{k+1})$.
	\end{proof}
	
	On the other hand, the opposite inclusion $\HS_i(I^{k+1})\subseteq I\cdot\HS_i(I^k)$ holds for all $k\gg0$, as we show next.
	
	\begin{Theorem}\label{Thm:HS(R(I))}
		Let $I\subset S$ be a monomial ideal. Then, $$\HS_i(I^{k+1})\subseteq I\cdot\HS_i(I^k)$$ for all $k\gg0$.
		
	\end{Theorem}
	\begin{proof}
		Firstly, we notice that $\Tor^S_i(K,\mathcal{R}(I))$ is a finitely generated bigraded module over $\mathcal{R}(I)\otimes_SK=\mathcal{R}(I)\otimes_SS/\m\cong\mathcal{R}(I)/\m\mathcal{R}(I)$, with $(*,k)$th graded component
		$$
		\Tor^S_i(K,\mathcal{R}(I))_{(*,k)}\cong\Tor^S_i(K,\mathcal{R}(I)_{(*,k)})\cong\Tor^S_i(K,I^k).
		$$
		
		Since the ring $\mathcal{F}_\m(I)=\mathcal{R}(I)/\m\mathcal{R}(I)=\bigoplus_{k\ge0}(I^k/\m I^k)$ is standard graded, then
		\begin{equation}\label{eq:finGenTor}
			\Tor^S_i(K,I^{k+1})\ =\ \mathcal{F}_\m(I)_{(*,1)}\cdot\Tor^S_i(K,I^{k})
		\end{equation}
		for all $k\gg0$ (see also \cite[Proposition 2.8]{FS2}). Equation (\ref{eq:finGenTor}) implies that
		$$
		\{{\bf a}\ :\ \Tor^S_i(K,I^{k+1})_{\bf a}\ne0\}\ \subseteq\ \{{\bf b+c}\ :\ (I/\m I)_{\bf b}\ne0,\ \Tor^S_i(K,I^{k})_{\bf c}\ne0\}
		$$
		for all $k\gg0$. Since $\HS_i(I^{k+1})=({\bf x^a}:\ \Tor^S_i(K,I^{k+1})_{\bf a}\ne0)$, the above inclusion implies that $\HS_i(I^{k+1})\subseteq I\cdot\HS_i(I^{k})$ for all $k\gg0$, as desired.
	\end{proof}
	
	Example \ref{Ex:HS-Not-Inclusion} shows that, in general, the inclusion $\HS_i(I^{k+1})\subseteq I\cdot\HS_i(I^k)$ is not an equality for all $k\gg0$.

	Nonetheless, if $I\subset S$ is a monomial ideal with \textit{eventually linear powers}, that is, if $I$ is generated in a single degree and $I^k$ has linear resolution for all $k\gg0$, then
	\begin{Proposition}
		Let $I\subset S$ be a monomial ideal with eventually linear powers. Then
		$$\HS_i(I^{k+1})=I\cdot\HS_i(I^k)$$ for all $k\gg0$.
	\end{Proposition}
	\begin{proof}
		By Theorem \ref{Thm:HS(R(I))} we have $\HS_i(I^{k+1})\subseteq I\cdot\HS_i(I^k)$ for all $k\gg0$. To show the opposite inclusion, let $\ell$ large enough such that $I^k$ has linear resolution for all $k\ge\ell$, and let ${\bf x^b}\in\mathcal{G}(I)$, ${\bf x^a}\in\mathcal{G}(\HS_i(I^k))$. We will show that ${\bf x^{a+b}}\in\HS_i(I^{k+1})$.
		
		The short exact sequence
		$$
		0\rightarrow {\bf x^b}I^k\rightarrow I^{k+1}\rightarrow I^{k+1}/({\bf x^b}I^k)\rightarrow0
		$$
		induces the exact sequence
		$$
		\Tor^S_{i+1}(K,I^{k+1}/({\bf x^b}I^k))_{\bf a+b}\rightarrow\Tor^S_i(K,{\bf x^b}I^k)_{\bf a+b}\rightarrow\Tor^S_i(K,I^{k+1})_{\bf a+b}.
		$$
		Let $d=\alpha(I)$. By assumption $I^k$ and $I^{k+1}$ have $dk$ and $d(k+1)$-linear resolutions, respectively. In particular, $\alpha(I^{k+1})=d(k+1)$ and so $\alpha(I^{k+1}/({\bf x^b}I^k))\ge d(k+1)$. It follows that $\Tor^S_{i+1}(K,I^{k+1}/({\bf x^b}I^k))_j=0$ for $j\le d(k+1)+i$. Since ${\bf x^a}\in\mathcal{G}(\HS_i(I^k))$ we have $|{\bf a}|=dk+i$. Moreover $|{\bf b}|=d$ because $I$ is equigenerated in degree $d$. Therefore $\Tor^S_{i+1}(K,I^{k+1}/({\bf x^b}I^k))_{\bf a+b}=0$ because $|{\bf a+b}|=d(k+1)+i$. Hence, we have an injective map
		$$
		0\rightarrow\Tor^S_i(K,{\bf x^b}I^k)_{\bf a+b}\rightarrow\Tor^S_i(K,I^{k+1})_{\bf a+b}.
		$$
		Notice that $\Tor^S_i(K,{\bf x^b}I^k)_{\bf a+b}\cong\Tor^S_i(K,I^k)_{\bf a}(-{\bf b})\ne0$ since ${\bf x^a}\in\mathcal{G}(\HS_i(I^k))$. Hence $\Tor^S_i(K,I^{k+1})_{\bf a+b}\ne0$ as well, and this implies that ${\bf x^{a+b}}\in\HS_i(I^{k+1})$.
	\end{proof}
	
	If $I$ has linear powers, then the previous result implies that $I\cdot\HS_i(I^k)\subseteq\HS_i(I^{k+1})$ for all $k\ge1$, and by Theorem \ref{Thm:HS(R(I))} equality holds for all $k\gg0$. Hence, as an immediate consequence we have
	\begin{Theorem}\label{Thm:HS-lp}
		Let $I\subset S$ be a monomial ideal with linear powers. Then $\HS_i(\mathcal{R}(I))$ is a finitely generated bigraded $\mathcal{R}(I)$-module.
	\end{Theorem}
	
	We close this section by describing another situation which ensures that the equality $\HS_i(I^{k+1})=I\cdot\HS_i(I^k)$ holds for all $k\gg0$ and some $i$.
	\begin{Proposition}\label{Prop:HS-n-1}
		Let $I\subset S$ be a monomial ideal generated in a single degree $d$. Assume that the socle of $I^k$, $\textup{soc}(I^k)$, is generated in degree $dk-1$ for all $k\gg0$. Then $\HS_{n-1}(I^{k+1})=I\cdot\HS_{n-1}(I^k)$ for all $k\gg0$.
	\end{Proposition}
	\begin{proof}
		Since, as we recalled before, $\HS_{n-1}(I^k)=x_1x_2\cdots x_n\cdot\textup{soc}(I^k)$ for all $k\ge1$, it is enough to show that $\soc(I^{k+1})=I\cdot\soc(I^k)$ for all $k\gg0$. By Theorem \ref{Thm:HS(R(I))} we have $\HS_{n-1}(I^{k+1})\subseteq I\cdot\HS_{n-1}(I^k)$ for all $k\gg0$. Hence $\soc(I^{k+1})\subseteq I\cdot\soc(I^k)$ for all $k\gg0$. To show the opposite inclusion, let $u\in\mathcal{G}(I)$ and $v\in\mathcal{G}(\soc(I^k))$ where $k$ is large enough. Our assumptions ensure that $\deg(v)=dk-1$. Now $uv\in I(I^k:\m)\subseteq(I^{k+1}:\m)$. Since $\deg(uv)=d(k+1)-1$, it follows that $uv\notin I^{k+1}$ because $I^{k+1}$ is equigenerated in degree $d(k+1)$. Hence $uv\in\textup{soc}(I^{k+1})$. This shows that $\soc(I^{k+1})=I\cdot\soc(I^k)$ for all $k\gg0$, as desired.
	\end{proof}
	
	We apply this result to edge ideals. Let $G$ be a finite simple graph on the vertex set $[n]=\{1,\dots,n\}$. The \textit{edge ideal} of $G$ is the monomial ideal $I(G)\subset S$ generated by the monomials $x_ix_j$ for which $\{i,j\}$ is an edge of $G$.
	
	\begin{Corollary}
		Let $G$ be a connected graph on the vertex set $[n]$ such that any two induced odd cycles of $G$ have distance at most one, that is, there is an edge between two vertices of the odd cycles, or the odd cycles share a common vertex. Then $\HS_{n-1}(I(G)^{k+1})=I(G)\cdot\HS_{n-1}(I(G)^k)$ for all $k\gg0$.
	\end{Corollary}
	\begin{proof}
		It is proved in \cite[Proposition 2.2]{CHL} that $\soc(I(G)^k)$ is equigenerated in degree $2k-1$ for all $k\ge1$. Hence, the assertion follows from Proposition \ref{Prop:HS-n-1}.
	\end{proof}
	
	A basic result in graph theory says that a graph is bipartite if and only if it does not contain any induced odd cycle. Therefore, the previous result applies to all bipartite graphs, as well as to all unicyclic graphs. That is, graphs containing only one induced cycle.
	
	\section{Asymptotic behavior of $\HS_i(I^k)$}\label{sec2}
	
	Let $I\subset S$ be a graded ideal and let $\p\in\Ass(I)$ be an associated prime ideal of $I$. Recall that the \textit{Castelnuovo-Mumford regularity} of $I$ is
	$$
	\reg(I)\ =\ \max\{j\ :\ \Tor^S_i(K,I)_{i+j}\ne0\},
	$$
	the \textit{$\v_\p$-number} of $I$ is
	$$
	\v_\p(I)\ =\ \min\{\deg(f)\ :\ f\in S_d,\ (I:f)=\p\},
	$$
	and the \textit{$\v$-number} of $I$ is
	$$
	\v(I)\ =\ \min_{\p\in\Ass(I)}\v_\p(I).
	$$
	
	\begin{Theorem}\label{Thm:AsymHS}
		Let $I\subset S$ be a monomial ideal with linear powers. The following statements hold.
		\begin{enumerate}
			\item[\textup{(a)}] The set $\Ass(\HS_i(I^k))$ stabilizes: $\Ass(\HS_i(I^{k+1}))=\Ass(\HS_i(I^{k}))$ for $k\gg0$. We denote the common sets $\Ass(\HS_i(I^k))$ for $k\gg0$ by $\Ass^\infty_i(I)$.
			\item[\textup{(b)}] For all $k\gg0$, we have $\depth\, S/\HS_i(I^{k+1})=\depth\,S/\HS_i(I^k)$.
			\item[\textup{(c)}] For all $k\gg0$, $\reg\,\HS_i(I^k)$ is a linear function in $k$.
			\item[\textup{(d)}] For all $k\gg0$, $\v(\HS_i(I^k))$ is a linear function in $k$.
			\item[\textup{(e)}] For all $k\gg0$ and $\p\in\Ass^\infty_i(I)$, $\v_\p(\HS_i(I^k))$ is a linear function in $k$.
		\end{enumerate}
	\end{Theorem}
	\begin{proof}
		By Theorem \ref{Thm:HS-lp}, $\HS_i(\mathcal{R}(I))$ is a finitely generated bigraded module over the standard graded ring $\mathcal{R}(I)$ with $\mathcal{R}(I)_{0}=S$. Thus, statement (a) is a consequence of \cite[Proposition 2]{ME79}. By \cite[Theorem 1.1]{HH2005}, $\depth\,\HS_i(\mathcal{R}(I))_{(*,k)}=\depth\,\HS_i(I^k)$ is constant for all $k\gg0$. Since $\depth\,S/\HS_i(I^k)=\depth\,\HS_i(I^k)-1$ for all $k$, statement (b) follows. For the proof of the  statements (c), (d) and (e), we recall that by Theorem \ref{Thm:HS-lp} we have that
		\begin{equation}\label{eq:HS-I-k_0}
			\HS_i(I^k)\ =\ \mathcal{R}(I)_{(*,k-k_0)}\cdot\HS_i(I^{k_0})\ =\ I^{k-k_0}\cdot\HS_i(I^{k_0})
		\end{equation}
		for some $k_0>0$ and all $k\ge k_0$. Set $M=\HS_i(I^{k_0})$. Since $M$ is a finitely generated bigraded $\mathcal{R}(I)$-module, then \cite[Theorem 3.2]{TW2005} implies that $\reg(I^kM)$ is an eventually linear function in $k$. From equation (\ref{eq:HS-I-k_0}) we conclude that $\reg\,\HS_i(I^k)$ is an eventually linear function in $k$ as well, and (c) follows. By \cite[Theorem 2.8 (2)-(3)]{Ghosh24} it follows that $\v_\p(I^kM)$, for any prime $\p$ with $\p\in\Ass(I^kM)$ for all $k\gg0$, and $\v(I^kM)$, are eventually linear functions in $k$ for all $k\gg0$. Using again equation (\ref{eq:HS-I-k_0}) we see that the statements (d) and (e) indeed hold.
	\end{proof}
	
	At the moment, we do not know whether Theorem \ref{Thm:AsymHS} holds for any monomial ideal $I\subset S$.
	
	The following result complements Theorem \ref{Thm:AsymHS}.
	\begin{Proposition}
		Let $I\subset S$ be a monomial ideal with linear powers. Then,
		$$
		\lim_{k\rightarrow\infty}\depth\,\HS_i(I^k)\ \le\ \dim\HS_i(\mathcal{R}(I))-\dim\frac{\HS_i(\mathcal{R}(I))}{\mathcal{R}(I)_{(*,1)}*\HS_i(\mathcal{R}(I))},
		$$
		and equality holds if $\HS_i(\mathcal{R}(I))$ is Cohen-Macaulay.
	\end{Proposition}
	\begin{proof}
		The result follows from Theorem \ref{Thm:HS(R(I))} and \cite[Theorem 1.1]{HH2005}.
	\end{proof}
	
	To illustrate the theory as developed thus far, we investigate the homological shift ideals of powers of three special families of monomial ideals.
	
	\subsection{Monomial ideals in two variables}\label{subsec2.1}
	
	Let $S=K[x,y]$. The set of monomial ideals $I\subset S$ is in bijection with the set of all pairs $({\bf a},{\bf b})\in\ZZ_{\ge0}^m\times\ZZ_{\ge0}^m$ satisfying
	\begin{equation}\label{eq:a-b}
		{\bf a}:a_1>a_2>\cdots > a_m\ge 0\,\,\,\,\,\,\,\,\textup{and}\,\,\,\,\,\,\,\,{\bf b}:0\le b_1<b_2<\cdots<b_m.
	\end{equation}
	Indeed, if $I$ is a monomial ideal of $S$, then $\mathcal{G}(I)=\{x^{a_1}y^{b_1},x^{a_2}y^{b_2},\dots,x^{a_m}y^{b_m}\}$ for two sequences ${\bf a}$ and ${\bf b}$ as above. Conversely, if ${\bf a}$ and ${\bf b}$ are as in (\ref{eq:a-b}), then $x^{a_1}y^{b_1},x^{a_2}y^{b_2},\dots,x^{a_m}y^{b_m}$ is the minimal generating set of a monomial ideal $I\subset S$.
	
	Hereafter, we set $I_{\bf a,b}=(x^{a_1}y^{b_1},x^{a_2}y^{b_2},\dots,x^{a_m}y^{b_m})$, where ${\bf a}$ and ${\bf b}$ are as in (\ref{eq:a-b}). Moreover, we put $\p_x=(x)$, $\p_y=(y)$ and $\m=(x,y)$.
	
	\begin{Proposition}\label{Prop:HS-I-a-b}
		Let $I=I_{\bf a,b}$ be a monomial ideal in two variables. Then,
		\begin{enumerate}
			\item[\textup{(a)}] $\HS_0(I)=I$, $\HS_1(I)=(x^{a_j}y^{b_{j+1}}:1\le j\le m-1)$ and $\HS_i(I)=0$ for $i\ne0,1$.
			\item[\textup{(b)}] We have
			$$
			\depth\,S/\HS_i(I^k)\ =\ \begin{cases}
				1&\textit{for all}\,\ k\ge1,\,\ \textit{if}\ m=1,\\
				0&\textit{for all}\,\ k\ge2,\,\ \textit{if}\ m=2,\\
				0&\textit{for all}\,\ k\ge1,\,\ \textit{if}\ m\ge3.
			\end{cases}
			$$
			\item[\textup{(c)}] \textup{(c1)} For $i=0,1$, we have $\p_x\in\Ass_i^\infty(I)$ if and only if $a_m>0$.\\
			\textup{(c2)} For $i=0,1$, we have $\p_y\in\Ass_i^\infty(I)$ if and only if $b_1>0$.\\
			\textup{(c3)} For $i=0,1$, we have $\m\in\Ass_i^\infty(I)$ if and only if $m>0$.
		\end{enumerate}
	\end{Proposition}
	\begin{proof}
		Part (a) is known, for instance, see \cite[Lemma 3.3]{Fic24}. Part (b) follows from the quoted lemma and the Auslander-Buchsbaum formula. Finally, part (c) follows from (a) and \cite[Proposition 5.1]{FS2}.
	\end{proof}
	
	To compute $\v(\HS_i(I_{\bf a,b})^k)$ one can use \cite[Theorem 5.6]{FS2}.
	
	\subsection{Principal Borel ideals}
	
	Recall that a monomial ideal $I\subset S$ is called \textit{strongly stable} if for all $u\in\mathcal{G}(I)$ and all $i<j$ with $x_j$ dividing $u$ we have $x_i(u/x_j)\in I$. The smallest strongly stable containing $u_1,\dots,u_m$ is denoted by $I=B(u_1,\dots,u_m)$ and the monomials $u_1,\dots,u_m$ are called the \textit{Borel generators} of $I$. If $I=B(u)$ has only one Borel generator, then $I$ is called a \textit{principal Borel ideal}.
	
	Given monomials $u,v\in S$, we set $u:v=u/\textup{gcd}(u,v)$. Thus $(u_1,\dots,u_{j-1}):(u_j)$ is generated by the monomials $u_i:u_j$ for $i=1,\dots,j-1$.
	
	Recall that a monomial ideal $I\subset S$ has \textit{linear quotients} if there exists an order $u_1,\dots,u_m$ of $\mathcal{G}(I)$ such that the ideals $(u_1,\dots,u_{j-1}):u_j$ are generated by variables for all $j=2,\dots,m$. We put $\set(u_j)=\{i\ :\ x_i\in(u_1,\dots,u_{j-1}):u_j\}$ for all $j$. If $I\subset S$ is equigenerated and it has linear quotients, then it has linear resolution.
	
	For a subset $F\subseteq[n]$ we set ${\bf x}_F=\prod_{i\in F}x_i$. In particular, we have ${\bf x}_{\emptyset}=1$. Then, by \cite[Lemma 1.5]{ET} we have $\pd\,I=\max\{|\set(u)|:u\in\mathcal{G}(I)\}$ and
	$$
	\HS_i(I)\ =\ ({\bf x}_Fu\ :\ u\in\mathcal{G}(I),\ F\subseteq\set(u),\ |F|=i).
	$$
	
	It is not difficult to see that any strongly stable ideal $I$ has linear quotients with respect to the lex order induced by $x_1>x_2>\dots>x_n$, and that for each $u\in\mathcal{G}(I)$ we have $\set(u)=\{1,\dots,\max(u)-1\}$, where $\max(u)=\max\{i\ :\ x_i\ \textup{divides}\ u\}$.
	
	\begin{Proposition}\label{Prop:HSiBorel}
		Let $I=B(u)$ be a principal Borel ideal. For all $k\ge1$, we have
		$$
		\HS_i(I^k)\ =\ (x_{j_1}\cdots x_{j_i}v\ :\ v\in\mathcal{G}(I^k),\, 1\le j_1<\dots<j_i<\max(v)).
		$$
		Moreover, for all $0\le i<\max(u)$ we have $\reg\,\HS_i(I^k)=\deg(u)k+i$ for all $k\ge1$, and $\depth\,S/\HS_i(I^k)=n-\max(u)$ for all $k\ge2$.
	\end{Proposition}
	\begin{proof}
		It is easily seen that $B(u)^k=B(u^k)$ for all $k\ge1$. Thus, the formula given for $\HS_i(I^k)$ follows from this observation and the previous discussion. In view of this fact, it is enough to consider the case $k=1$. In \cite[Theorem 3.4]{BJT019}, Bayati \textit{et al.} proved that $\HS_i(B(u))$ has linear quotients with respect to the lex order. From this observation, the formula for the regularity follows. To prove the formula for the depth, it is enough to show that $\pd\,\HS_i(B(u^k))=\max(u^k)-1=\max(u)-1$ for all $0\le i<\max(u)$ and all $k\ge2$. Since $\HS_i(B(u^k))$ has linear quotients with respect to the lex order, \cite[Lemma 2]{F2} gives $\set(v)\subseteq\{1,\dots,\max(v)-1\}$ for all $v\in\HS_i(B(u^k))$. This shows that $\pd\,\HS_i(B(u^k))\le\max(u^k)-1$. To prove the opposite inequality, we let $v=(\prod_{\ell=1}^ix_{\max(u)-\ell})u^k\in\mathcal{G}(\HS_i(B(u^k)))$ and show that $\set(v)=\{1,\dots,\max(u)-1\}$. Let $j\in\set(v)$. If $j<\max(u)-i$, then $w=x_j(v/x_{\max(u)-1})=x_j(\prod_{\ell=2}^ix_{\max(u)-\ell})u^k\in\mathcal{G}(\HS_i(B(u^k)))$ is greater than $v$ with respect to lex order. Thus, $w:v=x_j$ and $j\in\set(v)$. Otherwise, if $j\ge\max(u)-i$, then $w=x_j(v/x_{\max(u)})=(\prod_{\ell=1}^ix_{\max(u)-\ell})(x_j(u^k/x_{\max(u)}))\in\mathcal{G}(\HS_i(B(u^k)))$ since $k\ge2$ and so $x_j(u^k/x_{\max(u)})\in\mathcal{G}(B(u^k))$ has maximum equal to $\max(u)$. Since $w$ is greater than $v$ with respect to lex order and $w:v=x_j$, we have $j\in\set(v)$.
	\end{proof}
	
	Given a monomial ideal $I$ and an integer $j$, let $I_{\ge j}=(u\in\mathcal{G}(I):|\supp(u)|\ge j)$. Here, the \textit{support} of a monomial $u\in S$ is the set $\supp(u)=\{x_i\ :\ x_i\ \textup{divides}\ u\}$.
	
	\begin{Corollary}\label{Cor:HS(m^k)}
		Let $\m=(x_1,\dots,x_n)$. Then $\HS_i(\m^k)=(\m^{k+i})_{\ge i+1}$ for all $k\ge1$ and all $0\le i\le n-1$. In particular, $\reg\,\HS_i(\m^k)=k+i$ and $\depth\,S/\HS_i(\m^k)=0$.
	\end{Corollary}
	\begin{proof}
		The assertion is immediate for $k=1$. For $k\ge2$, notice that $\m^k=B(x_n^k)$ and then apply Proposition \ref{Prop:HSiBorel}.
	\end{proof}
	Let ${\bf c}=(c_1,\dots,c_n)\in\mathbb{Z}_{\ge0}^n$.  A monomial $u= x_1^{a_1}\cdots x_n^{a_n}$ is called {\em ${\bf c}$-bounded} if $a_i\leq c_i$ for all $i$. A monomial ideal $I$ is called \textit{${\bf c}$-bounded strongly stable}, if each monomial in $\mathcal{G}(I)$ is ${\bf c}$-bounded and for all $u\in\mathcal{G}(I)$ and all integers $i<j$ for which $x_i$ divides $u$ and $x_i(u/x_j)$ is ${\bf c}$-bounded, it follows that  $x_i(u/x_j)\in I$. The smallest ${\bf c}$-bounded strongly stable ideal containing a ${\bf c}$-bounded monomial $u\in S$ is called a {\em ${\bf c}$-bounded principal Borel ideal} and is denoted by $B^{{\bf c}}(u)$.
	
	\begin{Proposition}\label{prop:cbounded}
		For any ${\bf c}$-bounded monomial $u$ and $k\geq 1$, the ideal $\HS_i(B^{{\bf c}}(u)^k)$ has linear quotients.
	\end{Proposition}
	\begin{proof}
		By \cite[Theorem 3.1]{AHZ2022}, $(B^{{\bf c}}(u))^k= B^{k{\bf c}}(u^k)$. Moreover, by \cite[Theorem 3.2]{HMRZ021a}, we know that $\HS_j( B^{k{\bf c}}(u^k))$ has linear quotients for all $j$, as required.
	\end{proof}
	\subsection{Monomial complete intersections}
	
	Recall that a sequence ${\bf f}:f_1,\dots,f_m$ in $S$ is called a \textit{regular sequence} if $f_i$ is a non-zero divisor on $S/(f_1,\dots,f_{i-1})$ for all $i$.
	
	Let ${\bf u}:u_1,\dots,u_m$ be a sequence of monomials of $S$. It is easily seen that ${\bf u}$ is a regular sequence (for any order) if and only if $\supp(u_i)\cap\supp(u_j)=\emptyset$ for all $i\ne j$. 
	
	We put ${\bf u^b}=\prod_{i=1}^mu_i^{b_i}$ and $|{\bf b}|=b_1+\dots+b_m$, for ${\bf b}=(b_1,\dots,b_m)\in\ZZ_{\ge0}^m$. Moreover, we set $\max{\bf u^b}=\max\{i:b_i\ne0\}$ if ${\bf b}\ne{\bf 0}=(0,\dots,0)$.
	
	\begin{Proposition}
		Let ${\bf u}:u_1,\dots,u_m$ be a monomial regular sequence in $S$. Let $I=(u_1,\dots,u_m)$ and $d_i=\deg(u_i)$ for all $i$, and assume that $d_1\le d_2\le\dots\le d_m$. Then, for all $k\ge1$ we have
		\begin{equation}\label{eq:HSiCI}
			\HS_i(I^k)\ =\ (u_{j_1}\cdots u_{j_i}{\bf u^b}\ :\ |{\bf b}|=k,\, 1\le j_1<\dots<j_i<\max{\bf u^b}).
		\end{equation}
		In particular, for all $k\ge1$ we have $\depth\,S/\HS_i(I^k)=n-m$ and for all $k\ge2$
		$$
		\reg\,\HS_i(I^k)\ =\ d_mk+\sum_{j=1}^{m}d_{j}+\sum_{j=1}^id_{m-j}-(m-1).
		$$
	\end{Proposition}
	\begin{proof}
		Let $R=K[y_1,\dots,y_m]$ be the non-standard graded polynomial ring with $\deg(y_i)=d_i$ for all $i$. The $K$-algebra homomorphism $\varphi:R\rightarrow S$ defined by setting $\varphi(y_i)=u_i$ is flat, because ${\bf u}$ is a regular sequence. Let $\n=(y_1,\dots,y_m)$, then $\varphi(\n)=I$. Therefore, if $\FF$ is the minimal free $R$-resolution of $\n^k$, then $\FF\otimes_RS$ is the minimal free $S$-resolution of $I^k$. Hence $\varphi(\HS_i(\n^k))=\HS_i(I^k)$ for all $i$ and $k$.
		
		To compute $\HS_i(\n^k)$ we may employ the Eagon-Northcott complex. Alternatively, observe that $\n^k$ is equal to the principal Borel ideal $B(y_m^k)$. So, by Proposition \ref{Prop:HSiBorel},
		$$
		\HS_i(\n^k)\ =\ (y_{j_1}\cdots y_{j_i}{\bf y^b}\ :\ {\bf b}\in\ZZ_{\ge0}^m,\, |{\bf b}|=k,\, 1\le j_1<\dots<j_i<\max{\bf y^b}).
		$$
		Applying the map $\varphi$, the desired formula (\ref{eq:HSiCI}) follows.
		
		By Corollary \ref{Cor:HS(m^k)}, we have $\pd\,S/\HS_i(\n^k)=m$ for $0\le i<m$ and $k\ge1$. Thus $\pd\,S/\HS_i(I^k)=m$, and so $\depth\,S/\HS_i(I^k)\,=n-m$ by applying the Auslander-Buchsbaum formula.
		
		By the proof of Proposition \ref{Prop:HSiBorel}, $\HS_i(\n^k)$ has linear quotients with respect to the lex order. Applying the flat map $\varphi$ we then see that the $j$th multigraded free module in the minimal resolution of $\HS_i(I^k)$ has basis $\mathcal{B}_j$ given by the elements
		$$
		{\bf u}_F(\prod_{\ell=1}^iu_{j_\ell}){\bf u^b}
		$$
		satisfying the side conditions
		$$
		|{\bf b}|=k,\ \ \ \ 1\le j_1<\dots<j_i<\max{\bf u^b},\ \ \ \ F\subseteq\set((\prod_{\ell=1}^iu_{j_\ell}){\bf u^b}),\ \ \ \ |F|=j,
		$$
		where ${\bf u}_F=\prod_{i\in F}u_i$ for $F\ne\emptyset$, and ${\bf u}_\emptyset=1$. Hence,
		\begin{equation}\label{eq:reg-max}
			\reg\,\HS_i(I^k)\ =\ \max_{\bigcup_j\mathcal{B}_j}\,\{\deg({\bf u}_F(\prod_{\ell=1}^iu_{j_\ell}){\bf u^b})-j\}.
		\end{equation}
		
		Consider the monomial $w=(\prod_{\ell=1}^iy_{m-\ell})y_m^k\in\mathcal{G}(\HS_i(\n^k))$ where $k\ge2$. As in the proof of Proposition \ref{Prop:HSiBorel} we see that $\set(w)=\{1,\dots,m-1\}$. Applying the map $\varphi$ we see that the elements $${\bf b}_j\ =\ (\prod_{\ell=1}^ju_{m-\ell})(\prod_{\ell=1}^iu_{m-\ell})u_m^k$$ are basis elements of $\mathcal{B}_j$ for all $j=0,\dots,m-1=\pd\,\HS_i(I^k)$, with
		$$
		\deg{\bf b}_j-j\ =\ d_m k+\sum_{\ell=1}^{j}d_{m-\ell}+\sum_{\ell=1}^id_{m-\ell}-j.
		$$
		It is clear that $\deg{\bf b}_{m-1}-(m-1)$ maximizes (\ref{eq:reg-max}), and this proves the desired formula for the regularity.
	\end{proof}
	
	\section{Monomial ideals with homological linear powers}\label{sec3}
	
	It was expected for a brief time that if a monomial ideal has linear resolution, or it is equigenerated and has linear quotients, then the same would be true for each of its homological shift ideals. However, it was shown that this is not the case in general \cite{FH2023}. One then could ask the following
	\begin{Question}\label{Quest:Pow}
		Let $I\subset S$ be an equigenerated monomial ideal with linear powers or whose all powers have linear quotients. Does $\HS_i(I^k)$ have linear resolution or linear quotients for all $i$ and all $k\gg0$?
	\end{Question}
	
	In this regard, we have the following consequence of \cite[Theorem 1.3]{FH2023} proved by Herzog and the first author.
	\begin{Corollary}\label{Cor:FicHer}
		Let $I\subset S$ be an equigenerated monomial ideal whose all powers have linear quotients. Then, $\HS_1(I^k)$ has linear quotients for all $k\ge1$, in particular it has a linear resolution.
	\end{Corollary}
	
	Next, we show with an example that Question \ref{Quest:Pow} has a negative answer in general. For this aim, we first recall the following consequence of \cite[Lemma 1.1]{FH2023}.
	\begin{Corollary}\label{Cor:FicHer1}
		Let $I\subset S$ be a monomial ideal with linear resolution, and let $\mathcal{G}(I)=\{u_1,\dots,u_m\}$. Then
		$$
		(u_1,\dots,u_{i-1},u_{i+1},\dots,u_m):u_i
		$$
		is generated by variables for all $i=1,\dots,m$.
	\end{Corollary}
	
	Now, we are ready to present our example.
	\begin{Example}\label{Ex:I=B(125,333)}
		\rm Let $I=B(x_1x_2x_5,x_3^3)$ be a strongly stable ideal of $S=K[x_1,\dots,x_5]$ with two Borel generators. Since $I^k$ is also a strongly stable ideal for all $k\ge1$, it follows that $I^k$ has linear quotients with respect to the lex order, in particular it has a $3k$-linear resolution. Therefore, $\HS_2(I^k)$ is equigenerated in degree $3k+2$. Notice that the monomial $u=x_1x_2x_3^{3k}$ is a minimal generator of $\HS_2(I^k)$. Indeed $x_3^{3k}$ is a minimal generator of $I^k$ and $\set(x_3^{3k})=\{1,2\}$. Let $J$ be the monomial ideal generated by the set $\mathcal{G}(\HS_2(I^k))\setminus\{u\}$. If $\HS_2(I^k)$ were to have linear resolution, then by Corollary \ref{Cor:FicHer} the ideal $J:u$ would be generated by variables. Notice that the monomial $(x_1x_2x_5)x_3^{3(k-1)}$ is a minimal generator of $I^k$. Thus $v=x_3x_4(x_1x_2x_5)x_3^{3(k-1)}$ is a minimal generator of $\HS_2(I^k)$. Then $v:u=x_4x_5\in J:u$. We claim that neither $x_4$ nor $x_5$ belong to $J:u$, so that $x_4x_5$ is a minimal generator of $J:u$ and in view of what said above $\HS_2(I^k)$ does not have linear resolution. Suppose by contradiction that $x_4\in J:u$, then for some $j\ne 4$, we would have that $g=x_4(u/x_j)$ is a minimal generator of $\HS_2(I^k)$. Since $4=\max(g)$, it would follow that $g=x_{i_1}x_{i_2}w$ where $i_1<i_2<4$ and $w=x_4u/(x_jx_{i_1}x_{i_2})\in\mathcal{G}(I^k)$. Thus, $w$ should be one of the monomials $x_3^{3k-1}x_4$ or $x_2x_3^{3k-2}x_4$ or $x_1x_3^{3k-2}x_4$. But none of these monomials belongs to $I^k$. Hence $x_4\notin J:u$. We can proceed similarly to show that $x_5\notin J:u$.
	\end{Example}
	
	In view of \cite[Theorem 1.3]{FH2023}, the following problem posed in \cite{FH2023} remains open.
	\begin{Question}
		Let $I\subset S$ be a monomial ideal with linear resolution. Is it true that $\HS_1(I)$ has linear resolution? Does $\HS_1(I)$ even has linear quotients?
	\end{Question}
	Our next result suggest that the above question has a positive answer. To prove it, we first recall the following result proved in \cite[Proposition 2.3]{HMRZ021a}.
	
	\begin{Proposition}
		Let $I\subset S$ be a monomial ideal generated in a single degree, and suppose that $I$ has linear relations. Then $\HS_1(I)$ is generated by the monomials $x_iu$ with $u\in\mathcal{G}(I)$ for which there exist $j\neq i$ and $v\in\mathcal{G}(I)$ such that $x_iu=x_jv$.
	\end{Proposition}
	
	Let $I$ be a monomial ideal generated in degree $d$. In \cite{BHZ}, authors defined a graph $G_I$ associated to $I$ as follows: set $V(G_I)=\mathcal{G}(I)$ and  $\{u,v\}$ is an edge of $G_I$ if and only if $\deg(\lcm(u,v))=d+1$. Moreover, for all $u,v\in\mathcal{G}(I)$, let $G_I^{(u,v)}$ be the induced subgraph of $G_I$ with vertex set $$V(G_I^{(u,v)})\ =\ \{w\ :\  w \text{ divides } \lcm(u,v)\}.$$
	
	The following result was shown in \cite[Corollary 1.2]{BHZ}.\smallskip
	\begin{Proposition}
		Let $I\subset S$ be a monomial ideal generated in a single degree. Then $I$ is linearly related if and only if for all $u, v \in  \mathcal{G}(I)$ there is a path in $G_I^{(u,v)}$ connecting $u$ and $v$.
	\end{Proposition}\smallskip
	
	The next result shows that if $I\subset S$ has linear resolution, at least we can say that $\HS_1(I)$ is linearly related.
	
	\begin{Theorem}\label{Thm:HS-1}
		Let $I$ be a linearly related monomial ideal generated in degree $d$. Then $\HS_1(I)$ is also linearly related. 
	\end{Theorem}
	
	\begin{proof}
		Let $J=\HS_1(I)$ and $\alpha, \beta \in J$ with $\alpha \neq \beta$. To show that $J$ is linearly related, due to \cite[Proposition 1.1]{BHZ}, it is enough to show that $\alpha$ and $\beta$ are connected by a path in $G_J^{(\alpha,\beta)}$. Using \cite[Proposition 2.3]{HMRZ021a}, for some integers $k$ and $\ell$, we can write $\alpha=x_ku$ and $\beta=x_\ell v$ where  $u,v\in\mathcal{G}(I)$. If $u=v$, then $k\neq\ell$ and $\alpha,\beta$ are connected by an edge in $G_J^{(\alpha,\beta)}$, as required. Now we may assume that $u\neq v$. 
		
		Using the assumption that $I$ is linearly related together with \cite[Corollary 1.2]{BHZ}, we obtain a path $P: u=w_0, w_1, \ldots, w_r=v$ from $u$ to $v$ in $G_I^{(u,v)}$. In addition, we can assume that $P$ is of minimal length. Since $\deg(\lcm(w_i, w_{i+1}))=d+1$ for all $i=0, \ldots, r-1$, there exists some  $a_i \neq b_i$ with $x_{a_i}w_i= x_{b_i}w_{i+1}$  for all $i=0, \ldots, r-1$. Using \cite[Proposition 2.3]{HMRZ021a} gives $x_{a_i}w_i=x_{b_i}w_{i+1}\in\mathcal{G}(\HS_1(I))$. Note that $b_i \neq a_{i+1}$ for all $i=0, \ldots, r-2$ because $P$ is of minimal length. Indeed, if $b_i=a_{i+1}$ for some $i$, then $x_{a_i}w_i=x_{b_i}w_{i+1}=x_{a_{i+1}}w_{i+1}=x_{b_{i+1}}w_{i+2}$, which show that $w_i$ and $w_{i+2}$ are adjacent in $G_I^{(u,v)}$, a contradiction to the minimality of $P$. 
		
		This shows that $P':x_{b_0}w_1,x_{b_1}w_2,\ldots,x_{b_{r-1}}w_r$ is a path in $G_J^{(x_{b_0}w_1,x_{b_{r-1}}w_r)}$. Recall that $x_{a_0}u=x_{b_0}w_1$ and $x_{b_{r-1}}v=x_{b_{r-1}}w_r$. Next, observe that either  $\alpha=x_{a_0}u $ or $\alpha$ and $x_{a_0}u$ are adjacent in $G_J$ because $x_k(x_{a_0}u)= x_{a_0}(\alpha)$. Similarly, either $\beta=x_{b_{r-1}}v$ or $\beta$ and $x_{b_{r-1}}v$ are adjacent in $G_J$. In both cases, using $P'$  one obtains a path connecting $\alpha$ and $\beta$ in $G_J^{(\alpha,\beta)}$.
	\end{proof}
	
	The ideal in Example \ref{Ex:I=B(125,333)} shows that the conclusion of the previous theorem is no longer valid for the higher homological shift ideals.\medskip
	
	Let $I\subset S$ be a monomial ideal. We say that $I$ has \textit{homological linear powers} if $I$ has linear powers and $\HS_i(I^k)$ has linear resolution for all $i$ and all $k\ge1$. Whereas, we say that $I$ has \textit{eventually homological linear powers} if $I$ has linear powers and $\HS_i(I^k)$ has linear resolution for all $i$ and all $k\gg0$. Similarly, we say that $I$ has \textit{homological linear quotients} if $I$ has linear quotients and $\HS_i(I)$ has linear quotients for all $i$.\medskip
	
	Now, we present some classes of monomial ideals with homological linear powers.
	
	\begin{Theorem}\label{Thm:familiesHLP}
		The following are monomial ideals with homological linear powers.
		\begin{enumerate}
			\item[\textup{(a)}] Monomial ideals in two variables having linear resolution.
			\item[\textup{(b)}] ${\bf c}$-bounded principal Borel ideals.
			\item[\textup{(c)}] Cover ideals of whisker graphs.
			\item[\textup{(d)}] Principal Borel ideals.
			\item[\textup{(e)}] Hibi ideals.
		\end{enumerate}
	\end{Theorem}
	\begin{proof}
		(a) Let $I\subset S=K[x,y]$ be a monomial ideal in two variables having linear resolution. It is clear that principal monomial ideals have homological linear powers. Therefore, we assume that $I$ has at least two generators. From the discussion in the Subsection \ref{subsec2.1} it is clear that we can write $I=uJ$ where $u=x^py^q$ is a monomial, $J=I_{\bf a,b}$, ${\bf a,b}$ are as in (\ref{eq:a-b}), $m>1$ and $a_m=b_1=0$. Since $\HS_i((uJ)^k)=u^k\HS_i(J^k)$ for all $i$ and $k$, it is enough to show that $J$ has homological linear powers. Suppose that $\alpha(J)=d$. From Proposition \ref{Prop:HS-I-a-b}(a), it follows that $J$ has homological linear resolution if and only if
		$$
		\begin{cases}
			a_j+b_j=d&\textup{for}\,\, j=1,\dots,m,\\
			a_j+b_{j+1}=d+1&\textup{for}\,\, j=1,\dots,m-1.
		\end{cases}
		$$
		
		We deduce that $a_j-a_{j+1}=b_{j+1}-b_j=1$ for $j=1,\dots,m-1$. Since $a_m=b_1=0$, we have $a_j=d-(j-1)$ and $b_j=j-1$ for all $j=1,\dots,m$. Hence $J=\m^d$, and so $J$ has homological linear powers by applying Corollary \ref{Cor:HS(m^k)}.\smallskip
		
		The statement (b) is proved in Proposition~\ref{prop:cbounded}, statement (c) is proved in \cite[Theorem 4.8]{CF2}, statement (d) follows from Proposition \ref{Prop:HSiBorel}, and statement (e) is shown in \cite[Corollary 4.11]{CF1}.
	\end{proof}
	
	\section{Asymptotic Golodness of $\HS_i(I^k)$}\label{sec4}
	
	Let $(R,\m,K)$ be a Noetherian local ring, or a standard graded $K$-algebra with graded maximal ideal $\m$, and let ${\bf x}=x_1,\ldots,x_n$ be a minimal (homogeneous) system of generators of $\m$. We denote by $(K(R), \partial)$ the Koszul complex of $R$ with respect to ${\bf x}$. Let  $Z(R)$, $B(R)$ and $H(R)$ denote the module of cycles, boundaries and the homology of $K(R)$.\smallskip
	
	The formal power series $P_R(t)= \sum_{i \geq 0} \dim_K \Tor_i^{R} (R/\m,R/\m) t^i$ is called the {\em Poincar\'{e} series} of $R$. Serre showed that $P_R(t)$ is coefficientwise bounded above by the rational series
	\[
	\frac{(1+t)^n}{1-t\sum_{i\geq 0}\dim_K H_i({\bf x};R)t^i},
	\]
	where $n=\dim R$. We say that $R$ is a {\em Golod} ring, if the Poincar\'{e} series $P_R(t)$ of $R$ coincides with this upper bound given by Serre.\medskip
	
	It was shown by Golod (see \cite{Golod} or also \cite[Def. 5.5 and 5.6]{AvramovKustinMiller}) that $R$ is a {\em Golod} ring if and only if for each subset $\mathcal{S}$ of homogeneous elements of $\bigoplus_{i=1}^nH_i(R)$  there exists a function $\gamma$, which is defined on the set of finite
	sequences of elements from $\S$ with values in $\m\oplus\bigoplus_{i=1}^nK_i(R)$, subject to the following conditions:
	\begin{enumerate}
		\item[(G1)] if $h\in\S$, then $\gamma(h)\in Z(R)$ and $h=[\gamma(h)]$,
		\item[(G2)] if  $h_1,\dots,h_m$ is a sequence in $\mathcal{S}$ with $m>1$,  then
		$$
		\partial\gamma(h_1,\ldots,h_m)=\sum_{\ell=1}^{m-1}\overline{\gamma(h_1,\ldots,h_\ell)}\gamma(h_{\ell+1},\ldots,h_m),
		$$
		where $\bar{a} = (-1)^{i+1}a$ for $a\in K_i(R)$.
	\end{enumerate}
	A function $\gamma$ satisfying the properties (G1)-(G2) is called a {\em Massey operation} on $\mathcal{S}$.
	
	\begin{Remark}\label{Rem:useful}
		Let $\gamma$ be a Massey operation on $\mathcal{S}$. Then for each sequence $h_1,\dots,h_m$ of elements from $\S$ we have $\gamma(h_1,\dots,h_m)\in K_{a}(R)$ for some $a\ge 2m-1$.
	\end{Remark}
	\begin{proof}
		We proceed by induction on $m$, with the base case $m=1$ being trivial. Now, let $m>1$. For each $j<m$ we may assume that $\gamma(h_{1},\dots,h_{j})\in K_{b}(R)$ for some $b\ge 2j-1$, for each $1\le j\le m-1$. Then, by the property (G2) we have that
		$$
		\partial\gamma(h_1,\ldots,h_m)=\sum_{\ell=1}^{m-1}\overline{\gamma(h_1,\ldots,h_\ell)}\gamma(h_{\ell+1},\ldots,h_m).
		$$
		belongs to $K_{b}(R)K_{c}(R)=K_{b+c}(R)$ for some $b\ge 2\ell-1$, $c\ge 2(m-\ell)-1$ with $1\le\ell\le m-1$. Hence $\gamma(h_1,\dots,h_m)\in K_{a}(R)$ with $a=b+c+1\ge 2m-1$.
	\end{proof}
	
	\medskip
	
	We say that an ideal $I\subset R$ is \textit{Golod} if the quotient ring $R/I$ is Golod. Our aim in this section is to prove the following:
	\begin{Theorem}\label{Thm:HS-i-Golod}
		Let $I$ be a monomial ideal of $S=K[x_1,\dots,x_n]$ having linear powers. Then $\HS_i(I^k)$ is a Golod ideal for all $i$ and all $k\gg0$.
	\end{Theorem}
	
	This result will be a consequence of the following more general statement which is a generalization of \cite[Theorem 4.1]{HWY}.
	\begin{Theorem}\label{Thm:M_k-Golod}
		Let $R$ be a regular local ring or the polynomial ring $K[x_1,\dots,x_n]$, $I\subset R$ be an ideal, and $M=\bigoplus_{k>0}M_k$ be a finitely generated graded $\mathcal{R}(I)$-module such that $M_k$ is a proper ideal of $R$ for all $k>0$. Then $M_k$ is a Golod ideal for all $k\gg0$.
	\end{Theorem}\medskip
	
	For the proof of this result, we need some preparation.\medskip
	
	As in the statement of the theorem, let $M=\bigoplus_{k\ge0}M_k$ be a finitely generated graded $\mathcal{R}(I)$-module such that $M_k\subset R$ is a proper ideal for all $k>0$. Therefore, $\mathcal{R}(I)_\ell M_k=I^\ell M_k\subseteq M_{k+\ell}$ for all $k$ and $\ell$. We claim that there is a natural graded $R$-module homomorphism
	$$
	\ast\ :\ \mathcal{R}(I)\otimes\bigoplus_{i=1}^n(\bigoplus_{k\ge0}K_i(R/M_k))\rightarrow\bigoplus_{i=1}^n(\bigoplus_{k\ge0}K_i(R/M_k)).
	$$
	
	To prove this claim, it suffices to describe the map $\ast$ for all $\ell$ and $k$
	$$
	\ast\ :\ \left\{ \begin{array}{ccc}  \mathcal{R}(I)_\ell\otimes K_i(R/M_k) & \rightarrow & K_i(R/M_{k+\ell}) \\
		a \otimes v & \mapsto & a \ast v.
	\end{array}
	\right.
	$$
	
	For this aim, note that $K_i(R/M_k)=K_i(R)/M_kK_i(R)$, and let $K_i(R)\to K_i(R/M_k)$ be the canonical epimorphism which assigns to any $u\in K_i(R)$ the residue class $u+M_kK_i(R)$. Then, if $a\in\mathcal{R}(I)_\ell=I^\ell$ and $v=u+M_kK_i(R)\in K_i(R/M_k)$, we
	set
	$$
	a\ast v :=au+I^\ell M_kK_i(R)\subseteq M_{k+\ell}K_i(R).
	$$
	
	It is routine to check that map is well defined, in the sense that it is independent from the particular choice of the representative $u$ of $v$. It is also trivial to show that for $a\in \mathcal{R}(I)_j=I^j$, $b \in \mathcal{R}(I)_\ell=I^\ell$ and $v \in K_i(R/M_k)$ we have $(ab)\ast v=a\ast(b\ast v)$.
	
	\begin{Lemma}\label{Lem:ast}
		Let $\ast$ be the map introduced before. Then
		\begin{enumerate}
			\item[\textup{(i)}] For all $\ell$ and $k$, $\ast$ induced a map
			\begin{align*}
				\mathcal{R}(I)_\ell\otimes H_i(R/M_k)\ &\rightarrow\ H_i(R/M_{k+\ell}),\\
				\quad a\otimes [z]\ &\mapsto\ a \ast [z] := [a \ast z].
			\end{align*}
			We will denote the image of this map by $I^\ell\ast H_i(R/M_k)\subseteq H_i(R/M_{k+\ell})$.\medskip
			\item[\textup{(ii)}] There exists an integer $s$ such that for all $k\ge s$, $i\ge1$, and $\ell\ge0$, we have
			$$
			\mathcal{R}(I)_\ell\ast H_i(R/M_{k})=I^\ell\ast H_i(R/M_{k})=H_i(R/M_{k+\ell}).
			$$
		\end{enumerate}
	\end{Lemma}
	\begin{proof}
		(i) Let $v=u+M_{k}K_i(R)\in K_i(R/M_{k})$ and $a\in\mathcal{R}(I)_\ell=I^\ell$. It is enough to show that $a*v$ is a cycle whenever $v$ is a cycle and that it is a boundary whenever $v$ is a boundary. For both claims, notice that $\partial(au)=a\partial(u)$.
		
		If $v\in Z_i(R/M_k)$ is a cycle, then $\partial(v)=0$. Hence $\partial(u)\in M_{k}K_{i-1}(R)$. Thus $\partial(au)=a\partial(u)\in I^\ell M_{k}K_{i-1}(R)\subseteq M_{k+\ell}K_{i-1}(R)$, so that $\partial(a\ast v)\in Z_i(R/M_{k+\ell})$ is again a cycle.
		
		Similarly, if $v\in M_{k}B_i(R)$ is a boundary, then $v=\partial(w)$ for some $w\in M_{k}K_{i+1}(R)$. As before, $\partial(a*w)=a*v$, and since $a*w\in M_{k+\ell}K_{i+1}(R)$, it then follows that $a*v\in M_{k+\ell}B_i(R)$ is a boundary, as desired.\smallskip
		
		(ii) The assumptions guarantee that $M_{k+1}=IM_{k}$ for all $k\gg0$. Since $M$ is a finitely generated graded $\mathcal{R}(I)$-module, the Koszul homology module of $M$, $H_i(M)=H_i(R)\otimes M$, is a finitely generated graded $\mathcal{R}(I)$-module with $k$th graded component $H_i(M)_k=H_i(M_k)$. Since $M_k$ is an ideal and $H_i(R/M_k)\cong H_{i-1}(M_k)$, it follows that there exists an integer $s_i$ such that $H_i(R/M_{k})=I^{k-s_i}\ast H_i(R/M_{s_i})$ for all $k\ge s_i$, where the multiplication $\ast$ is described in Lemma \ref{Lem:ast}(i). Therefore, setting $s=\max\{s_0,\dots,s_{n-1}\}$, the assertion (ii) follows.
	\end{proof}
	
	We are now ready for the proof of the main result.
	\begin{proof}[Proof of Theorem \ref{Thm:M_k-Golod}]
		In the case that $R$ is the polynomial ring, let $\widehat{R}$ be the $\m$-adic completion of $R$. Since $H(\widehat{R}/J\widehat{R})=H(R/J)$ for any ideal $J\subset R$, we can replace $R$ by its completion and so we may assume that $R$ is local.
		
		We claim the following: for any integer $r\geq 1$, there exists an integer $s_r$ such that for all $k\geq s_r$ and each homogeneous subset $\S\subset \bigoplus_{i=1}^nH_i(R/M_k)$ there exists a function $\gamma$, defined on the set of sequences of elements of $\S$ of length $\le r$, such that the properties (G1) and (G2) hold.
		
		The claim will yield the desired result, since for any such function $\gamma$ and any $r\geq n/2+1$, we necessarily have $\gamma(h_1,\ldots,h_r)=0$ by Remark \ref{Rem:useful}.
		
		We will prove the claim by induction on $r$. For $r=1$, we may choose $s_r=1$ and the assertion is trivial.
		
		Let $r\geq 1$ and assume that the claim holds for all integers $\leq r$. By Lemma \ref{Lem:ast}(ii) there exists an integer $s$ such that $I^{k-s}\ast H_i(R/M_{s})=H_i(R/M_{k})$ for all $k\geq s$ and all $i>0$. Set $s_{r+1}=\max\{s_r,s\}+1$. Let $\mathcal{G}=\{g_1,\ldots,g_t\}$ be a homogeneous $K$-basis of $\bigoplus_{i=1}^nH_i(R/M_{s_r})$. By induction hypothesis there exists a function $\gamma$, defined on  the set of sequences of elements of $\mathcal{G}$ of length $\leq r$, such that (G1) and (G2) hold.
		
		Let $k\ge s_{r+1}$ and let $\S\subset \bigoplus_{i=1}^nH_i(R/M_k)$ be any set of homogeneous elements. We are going  to define a function $\gamma$ on sequences from $\S$ of length $r+1$ satisfying the conditions (G1) and (G2).
		
		First, let $h_1,\ldots,h_m$ be any sequence of elements $h_i\in\S$ of length $m\leq r$. Since $s_{r+1}>s$, each $h_i$ can be written as
		$h_i=\sum_{j=1}^ta_{ij}g_j$ with $a_{ij}\in I^{k-s_r}$ and where all $g_j\in H_\ell(R/M_{s_r})$ if $h_i\in H_\ell(R/M_k)$. We define the map $\gamma$ on this sequence by multi-linear extension. That is, we set
		$$
		\gamma(h_1,\ldots,h_m)\ =\ \sum_{j_1=1}^t\sum_{j_2=1}^t\cdots \sum_{j_m=1}^t(a_{1j_1}a_{2j_2}\cdots a_{mj_m})\ast\gamma(g_{j_1},\ldots,g_{j_m}),
		$$
		where we regard $a_{1j_1}\cdots a_{mj_m}$ as an element of $I^{k-s_r}$, so that $\gamma(h_1,\ldots,h_m)\in K(R/M_k)$.\smallskip
		
		Now we verify that $\gamma$ satisfies (G1) and (G2) on sequences from $\S$ of length $\leq r$.\smallskip
		
		\textit{Proof of} (G1): Let $h\in\S$ with presentation $h=\sum_{j=1}^t a_{j}g_j$ where  $a_j\in I^{k-s_r}$ for all $j$.
		Since $\gamma(g_j)\in Z(R/M_{s_r})$ then $a_{j} \ast \gamma(g_j)\in Z(R/M_{k})$, and hence we see
		that $\gamma(h)=\sum_{j=1}^t a_{j} \ast \gamma(g_j)$ belongs to $Z(R/M_{k})$ as well. Moreover, we have
		$$
		[\gamma(h)]=\sum_{j=1}^t [a_{j} \ast \gamma(g_j)]= \sum_{j=1}^t a_{j} \ast [\gamma(g_j)] = \sum_{j=1}^t a_{j}g_j=h,
		$$
		as desired.\hfill$\square$\medskip
		
		\textit{Proof of} (G2): Since $\partial$ is linear with respect to $\ast$ we have
		$$
		\partial \gamma(h_1,\ldots,h_m)= \sum(\prod_{i=1}^ma_{ij_i}) \ast \partial \gamma(g_{j_1},\ldots,g_{j_m}),
		$$
		where the sum runs over all $1\le j_k\le t$. For each summand in $\partial \gamma(h_1,\ldots,h_m)$, by using
		$a \ast (v+w) = a\ast v + a \ast w$ and $(ab) \ast (vw) = (a \ast v) (b \ast w)$, we have
		\begin{align*}
			(\prod_{i=1}^ma_{ij_i}) \ast \partial \gamma(&g_{j_1},\ldots,g_{j_m})
			= (\prod_{i=1}^ma_{ij_i}) \ast \sum_{\ell=1}^{m-1} \overline{\gamma(g_{j_1},\ldots,g_{j_\ell})} \gamma(g_{j_{\ell+1}},\ldots,g_{j_m})\\
			&=\ \sum_{\ell=1}^{m-1} (\prod_{i=1}^ma_{ij_i}) \ast \big( \overline{\gamma(g_{j_1},\ldots,g_{j_\ell})} \gamma(g_{j_{\ell+1}},\ldots,g_{j_m})\big)\\
			&=\ \sum_{\ell=1}^{m-1}\big((\prod_{i=1}^\ell a_{ij_i}) \ast \overline{\gamma(g_{j_1},\ldots,g_{j_\ell})}\big) \big((\prod_{i=\ell+1}^ma_{ij_i}) \ast \gamma(g_{j_{\ell+1}},\ldots,g_{j_m})\big).
		\end{align*}
		Summing over all indices $j_t$ yields the desired identity (G2), as desired.\hfill$\square$\medskip
		
		In order to define $\gamma$ on sequences  of elements of $\S$ of length $r+1$, it suffices to show that for any sequence $h_1,\ldots,h_{r+1}$ of elements from $\S$, the element
		$$
		b = \sum_{\ell=1}^{r}\overline{\gamma(h_1,\ldots,h_\ell)}\gamma(h_{\ell+1},\ldots,h_{r+1}),
		$$
		is a boundary of $K(R/M_k)$. To this end, observe that $b$ is a linear combination of expressions of the form
		$$
		b'=\sum_{\ell=1}^{r}\overline{\gamma(g_{k_1},\ldots,g_{k_\ell})}\gamma(g_{k_{\ell+1}},\ldots,g_{k_{r+1}})
		$$
		with coefficients in $I^{(r+1)(k-s_r)}\subset I^{k-s_r}$. Let $a$ be the coefficient of $b'$. Since $b'$ is a boundary in $K(R/M_{s_r})$ and $a\in I^{k-s_r}$, it follows that $ab'$ is a boundary in $K(R/M_{k})$. Thus $b\in B(R/M_k)$, as desired.
	\end{proof}
	
	Finally, we can prove our final result in the paper.
	\begin{proof}[Proof of Theorem \ref{Thm:HS-i-Golod}]
		By Theorem \ref{Thm:HS-lp}, the $i$th homological shift algebra of $I$ is a finitely generated (bi)graded $\mathcal{R}(I)$-module for which $\HS_i(\mathcal{R}(I))_k=\HS_i(I^k)$ a proper ideal of $S$ for all $k>0$. So, the assertion follows from Theorem \ref{Thm:M_k-Golod}.
	\end{proof}
	
	In \cite{HHun}, Herzog and Huneke introduced the concept of \textit{strongly Golod} ideal. Let $I\subset S$ be a graded ideal. We denote by $\partial I$ the graded ideal generated  by all partial derivatives of the generators of $I$. The ideal $I$ is called \textit{strongly Golod} if $(\partial I)^2\subseteq I$.\smallskip
	
	By \cite[Theorem 1.1]{HHun} any strongly Golod ideal is Golod. It is easy to see that $I^k$ is strongly Golod for all $k\ge2$ \cite[Theorem 2.1.(d)]{HHun}. Therefore, one is led to ask whether $\HS_i(I^k)$ is strongly Golod as well for all monomial ideals $I\subset S$, all $k\gg0$ and all $i$. Next, we show that this is not the case in general.\smallskip
	
	Let $S=K[x,y]$ and $I=(x,y)$. Using Proposition \ref{Prop:HS-I-a-b} we obtain that
	$$
	\HS_1(I^k)=(x^ky,x^{k-1}y^2,\dots,xy^{k}).
	$$
	Then $\partial_y(x^ky)^2=x^{2k}\notin\HS_1(I^k)$. So $\HS_1(I^k)$ is not strongly Golod for all $k\ge1$. Nonetheless, since $I$ has linear powers, by Theorem \ref{Thm:HS-i-Golod} we have that $\HS_1(I^k)$ is Golod for all $k\gg0$. Indeed, in this example the ideal $\HS_1(I^k)$ is Golod for all $k\ge1$, because $\HS_1(I^k)=xy(x,y)^{k-1}$ has linear resolution, and ideals with linear resolutions are Golod \cite[Theorem 4]{HRW}.\bigskip

	\section{Declarations}
	\noindent\textbf{Competing interests}
	The authors have no competing interests to declare that are relevant to the content of this article.
	
	\noindent\textbf{Funding}
	A. Ficarra was partly supported by INDAM (Istituto Nazionale di Alta Matematica), and also by the Grant JDC2023-051705-I funded by\\ MICIU/AEI/10.13039/501100011033 and by the FSE+. A. A. Qureshi is supported by Scientific and Technological Research Council of Turkey T\"UB\.{I}TAK under the Grant No: 124F113, and is thankful to T\"UB\.{I}TAK for their support.

\end{document}